\documentclass[reqno]{amsart}

\usepackage[utf8]{inputenc}
\usepackage[toc,page]{appendix}
\usepackage{amsmath,amssymb,amsthm,amsfonts}
\usepackage{changepage}
\usepackage{url}
\usepackage{hyperref}
\usepackage[usenames,dvipsnames]{xcolor}

\usepackage[foot]{amsaddr}

\theoremstyle{plain}
\newtheorem{lem}{Lemma}[section]
\newtheorem{cor}[lem]{Corollary}
\newtheorem{prop}[lem]{Proposition}
\newtheorem{thm}[lem]{Theorem}

\theoremstyle{definition}
\newtheorem{defn}[lem]{Definition}
\newtheorem{ex}[lem]{Example}

\newtheorem{problem}[lem]{Problem}
\newtheorem{remark}[lem]{Remark}

\newtheorem{conjecture}[lem]{Conjecture}

\newcommand{\bbn}{\mathbb{N}}
\newcommand{\bbf}{\mathbb{F}}
\newcommand{\Fq}{\mathbb{F}_q}
\newcommand{\fr}{\bbf_2[x^2,x^3]}
\newcommand{\fx}{\bbf_2[x]}
\newcommand{\bc}{b_c(n)}
\newcommand{\bw}{b_w(n)}
\newcommand{\bt}{b_t(n)}
\newcommand{\bcn}[1]{b_c(#1)}
\newcommand{\bwn}[1]{b_w(#1)}
\newcommand{\btn}[1]{b_t(#1)}

\begin{document}

\title{On atomic density of numerical semigroup algebras}

\author[A.A.~Antoniou]{Austin~A.~Antoniou}
\address{Austin~A.~Antoniou,
    Department of Mathematics,
    The Ohio State University
}
\email{antoniou.6@osu.edu}

\author[R.A.C.~Edmonds]{Ranthony A.C.~Edmonds}
\address{Ranthony A.C.~Edmonds,
    Department of Mathematics,
    The Ohio State University
}
\email{edmonds.110@osu.edu}

\author[B.~Kubik]{Bethany Kubik}
\address{Bethany Kubik,
    Department of Mathematics and Statistics,
    University of Minnesota Duluth
}
\email{bakubik@d.umn.edu}

\author[C.~O'Neill]{Christopher O'Neill}
\address{Christopher O'Neill,
    Mathematics and Statistics Department,
    San Diego State University
}
\email{cdoneill@sdsu.edu}

\author[S.~Talbott]{Shannon Talbott}
\address{Shannon Talbott,
    Department of Mathematics and Computer Science,
    Moravian College
}
\email{talbotts@moravian.edu}


\begin{abstract}
A numerical semigroup $S$ is a cofinite, additively-closed subset of the nonnegative integers that contains $0$.  In this paper, we initiate the study of atomic density, an asymptotic measure of the proportion of irreducible elements in a given ring or semigroup, for semigroup algebras. It is known that the atomic density of the polynomial ring $\mathbb{F}_q[x]$ is zero for any finite field $\mathbb{F}_q$;  we prove that the numerical semigroup algebra $\mathbb{F}_q[S]$ also has atomic density zero for any numerical semigroup~$S$. We also examine the particular algebra $\mathbb{F}_2[x^2,x^3]$ in more detail, providing a bound on the rate of convergence of the atomic density as well as a counting formula for irreducible polynomials using M\"{o}bius inversion, comparable to the formula for irreducible polynomials over a finite field $\mathbb{F}_q$.
\end{abstract}

\maketitle

\section{Introduction}

In this paper, $\bbn$ denotes the set of nonnegative integers, and $\mathbb F_q$ denotes the field with $q$ elements.

A \emph{numerical semigroup} is a subset $S\subseteq \bbn$ that is closed under addition, has finite complement in~$\bbn$, and contains zero.  Every numerical semigroup admits a unique generating set that is minimal with respect to containment and, unless otherwise stated, whenever we write
\[
S = \langle n_1, \ldots, n_k \rangle
= \{a_1n_1 + \cdots + a_kn_k : a_i \in \bbn\},
\]
we assume $n_1, \ldots, n_k$ are the minimal generators of $S$.  The smallest possible value of $k$ is called the \emph{embedding dimension} of $S$.  We say $a \in \bbn$ is a \emph{gap} of $S$ if $a \notin S$, and the largest gap of $S$, denoted~$\mathsf F(S)$, is called the \emph{Frobenius number} of $S$.  For more background on numerical semigroups, see~\cite{numericalappl}.

Given a numerical semigroup $S$ and a field $\bbf$, the \emph{semigroup algebra} $\bbf[S]$ is the set
\[
\bbf[S] = \{a_0 + a_1x + \cdots + a_dx^d \in \bbf[x] : a_i = 0 \text{ whenever } i \notin S\} \subset \bbf[x]
\]
of polynomials consisting only of terms $x^i$ with $i \in S$.
Note that if $S = \bbn$, then $\bbf[S] = \bbf[x]$. Semigroup algebras, when viewed as quotients of toric ideals~\cite{cls}, are central to combinatorial commutative algebra~\cite{grobpoly} and arise in a host of statistical~\cite{markovbook,algmarkov} and computational~\cite{clo} applications.

A nonconstant polynomial $f(x)\in \bbf[S]$ is \emph{irreducible} if we cannot write $f(x) = a(x)b(x)$ such that $a(x), b(x) \in \bbf[S]$ both have positive degree.  A \emph{factorization} of $f$ is an expression for $f$ as a product of irreducible elements. Unless $S = \bbn$, there will exist elements in $\bbf[S]$ that admit nonunique factorization, as, for instance,
\[
x^{n_1n_2} = (x^{n_1})^{n_2} = (x^{n_2})^{n_1}
\]
since $x^{n_1}$ and $x^{n_2}$ are necessarily irreducible in $\bbf[S]$
(here, ``unique'' means up to reordering and up to associates).
In the case where $S$ is an arbitrary semigroup, the semigroup algebra $\bbf[S]$ need not be atomic; that is, at least one non-unit element of $\bbf[S]$ cannot be written as a product of finitely many irreducible elements.  For example, if~$\bbf$ is a field of finite characteristic, it is possible to construct a torsion free atomic monoid $S$ of rank at least $2$ so that $\bbf[S]$ is not atomic~ \cite{coykendall}.  When $S$ is a numerical semigroup, however, $\bbf[S]$ is always atomic, and moreover, every element of $\bbf[S]$ has only finitely many factorizations. To see this, note that $\bbf[S] \subset \bbf[x]$ is a subalgebra of a unique factorization domain.

The concept of nonunique factorization has been studied in a wealth of settings, such as block monoids and Krull monoids~\cite{lensetprogress}, which are central to additive combinatorics and the study of algebraic number fields~\cite{krullcombinatorialsurvey}.  Numerical semigroups~\cite{numericalsurvey}, as well as the algebras over them~\cite{baruccinsalg}, also make frequent appearances.  This is especially true of the semigroup algebra $\fr$, which regularly occurs as a counterexample~\cite{nsalg1,nsalg2}.
For a thorough introduction to this vast research area, see~\cite{nonuniq}.

In this paper, we consider semigroup algebras of the form $\Fq[S]$ for some finite field $\Fq$ and numerical semigroup $S$. The semigroup algebra $\Fq[S]$ is naturally a subalgebra of $\Fq[x]$, whose irreducible polynomials and factorization structure play a critical role in coding theory~\cite{berlekampcodingtheory,ecctheory1} and combinatorial design theory~\cite{designtestgeneration,sparsesensingfinitegeometry}.  Let
\[
\Fq[S]^{(n)} = \{f(x) \in \Fq[S] : \deg f(x) = n\}
\]
denote the set of all polynomials of degree $n$ in $\Fq[S]$, and let
\[
a_q^S(n)
= \#\{\text{irreducible } f(x) \in \Fq[S]^{(n)}\}
\]
denote the number of irreducible elements of $\Fq[S]$ of degree $n$.
We wish to examine the \emph{atomic density} of $\Fq[S]$, that is, the limiting value of $\rho_q^S(n)$ as $n \to \infty$, where
\[
\rho_q^S(n)
= \frac{a_q^S(n)}{|\Fq[S]^{(n)}|}
\]
denotes the density of irreducible polynomials of degree $n$ in $\Fq[S]$. If $S = \bbn$, so that $\Fq[S] = \Fq[x]$, then it is known that the atomic density is 0 (see Theorem~\ref{t:polyatomicdensity}).  This means that, in some  sense, ``most'' polynomials in $\Fq[x]$ are reducible.

The notion of atomic density has also been examined in the context of arithmetic congruence monoids~\cite{baginski2014arithmetic,acmperiodic}, which are multiplicative submonoids of $\mathbb Z_{\ge 1}$ of the form
\[
M_{a,b} = \{n \in \mathbb Z_{\ge 1} : n \equiv a \bmod b\}
\]
with $a^2 \equiv a \bmod b$.  Atomic density in this setting is defined as the limiting value as $n \to \infty$ of the proportion of elements of $M_{a,b} \cap [1,n]$ that are irreducible.
Note that both arithmetical congruence monoids and numerical semigroup algebras over finite fields possess a natural well-ordering under which each element succeeds all of its divisors, and the definitions of atomic density respect these orderings.  The same is not true, however, for subalgebras of a multivariate multivariate polynomial ring, which has many different term orders, and it is for this reason that we restrict our focus to numerical semigroup algebras.

Surprisingly, atomic density does not appear to have been previously considered in the context of numerical semigroup algebras.  In this paper, we prove the following generalization to numerical semigroup algebras.

\begin{thm}\label{t:atomicdensity}
Any numerical semigroup algebra $\Fq[S]$ has atomic density $0$; that is,
\[\lim_{n \to \infty} \rho_q^S(n) = 0.\]
\end{thm}

The paper is organized as follows. In Section~\ref{sec:numsemigpalg}, we prove Theorem~\ref{t:atomicdensity}, the main result of the paper, by splitting the irreducible elements of $\Fq[S]$ into finitely many classes based on their factorizations in~$\Fq[x]$ (see Proposition~\ref{p:factorclassification}).  In the remaining sections of the paper, we obtain more refined results for the semigroup algebra $\fr$.  In Section~\ref{sec:friendlyring}, we prove that the irreducible polynomials of $\fr$ can be partitioned into three distinct classes (see Proposition~\ref{p:lineartermfactors}) and use this to provide an upper bound on the convergence rate of the limit in Theorem~\ref{t:atomicdensity} (see Theorem~\ref{t:friendlymain}).  In Section~\ref{sec:counting}, we provide a formula for the number of irreducible polynomials of each degree in $\fr$ in terms of
the M\"obius function (see Lemma~\ref{l:frcount}), analogous to a well-known formula for the number of irreducible polynomials of each degree in $\bbf_2[x]$ that follows from the M\"obius inversion formula.

Throughout this project, we used the \texttt{GAP} package \texttt{numericalsgps}~\cite{numericalsgpsgap} from within \texttt{Sage}, with the help of the \texttt{numsgpsalg}~\cite{numsgpsalg} and \texttt{NumericalSemigroup.sage}~\cite{numericalsgpssage} packages, both available on Github.

\section{Atomic density of numerical semigroup algebras} \label{sec:numsemigpalg}

Let $q$ be a prime power and let $S$ be a numerical semigroup.
In this section, we show that the asymptotic density of $\Fq[S]$ is zero.
We begin by viewing the irreducible polynomials of $\Fq[S]$ as elements of $\Fq[x]$ and by characterizing their factorizations in $\Fq[x]$ (see Proposition~\ref{p:factorclassification}).
In particular, we establish a bound on the factorization length that depends only on $q$ and $S$ (and \textit{not} the individual polynomials involved).
Then, we repeatedly apply Theorem~\ref{t:polyatomicdensity} to show that the proportion of irreducible polynomials in $\Fq[S]$ of degree $n$ shrinks with $n$.

We begin with two technical lemmas.  The first is of extremal combinatorial flavor and asserts that any polynomial $f \in \Fq[x]$ with sufficiently many factors is divisible by some polynomial which has only a constant term and ``high-degree'' terms (and so consequently lives in $\Fq[S]$).  The second implies that the set of polynomials in $\Fq[S]$ with nonzero constant term is divisor closed in $\Fq[x]$, that is, if $f\in x\Fq[S]+\Fq^\times$ and $g\in \Fq[x]$ divides $f$, then we actually have $g\in x\Fq[S]+\Fq^\times$.

\begin{lem}\label{lem:producing semigroup factor}
Let $q$ be a prime power, let $N \in \bbn$ be a positive integer and let $k \ge q^{N-1}$.
For any list of polynomials $f_1, \ldots, f_k \in \Fq[x]$ with each satisfying $f_i(0) \ne 0$, there is a subproduct $g$ of $f_1 \cdots f_k$ with $g \in \Fq + x^N \Fq[x]$.
That is, there exist $1 \le i_1 < \cdots < i_\ell \le k$ with $f = f_{i_1}\cdots f_{i_\ell} \in \Fq + x^N \Fq[x]$.
\end{lem}

\begin{proof}
If $N=1$, then $\Fq + x^N\Fq[x] = \Fq[x]$, so the statement is trivial.
Suppose, by way of induction, that the statement of the lemma is true for some $N\ge 1$, and let $f_1,\dots, f_k\in \Fq[x]$ with $k \ge q^N$ and $f_i(0)\neq0$ for each $i\le k$.

Since $k \ge q q^{N-1}$, we can apply the inductive hypothesis $q$ times to find $g_1, \ldots, g_q \in \Fq + x^N \Fq[x]$ with $g_1\cdots g_q \mid f_1\cdots f_k$.
To be precise, we treat the $q^N$ polynomials as $q$ separate collections of $q^{N-1}$ polynomials, applying the inductive hypothesis to each collection.
Notice, since each $f_i(0) \neq 0$, then each of the $g_i$ also has nonzero constant term.
Replacing $g_i$ with $g_i/g_i(0)$ where needed, we may assume that $g_i(0)=1$ for every $i$.

Let $c_i$ be the $x^N$ coefficient of $g_i$, so $g_i \equiv c_i x^N + 1 \pmod {x^{N+1}}$.
Notice that whenever $1 \le s < t \le q$,
\[
g_s\cdots g_t \equiv (c_s x^N + 1)\cdots(c_t x^N + 1) \equiv \left(\sum_{i=s}^t c_i\right) x^N + 1 \pmod {x^{N+1}}.
\]
Now, if any of the $q$ sums $c_1+\cdots+c_t$ for $t\in [1,q]$ is zero, we see that $g_1 \cdots g_t \in \Fq + x^{N+1}\Fq[x]$.
On the other hand, if none of these sums is zero, then two of the sums must be the same, meaning for some $s < t$, we have $c_1+\cdots+c_s = c_1+\cdots+c_t$.
This implies
\[c_{s+1} + \cdots + c_t = (c_1 + \cdots + c_t) - (c_1 + \cdots + c_s) = 0,\]
which means $g_{s+1}\cdots g_t$ has no nonconstant term of degree less than $N+1$.
In either case, we have found a subproduct in $ \Fq + x^{N+1}\Fq[x]$ that divides $g_1\cdots g_q$ and $f_1\cdots f_k$, as desired.
\end{proof}

\begin{lem}\label{lem:divisor closedness}
Let $\bbf$ be a field and $S$ be a numerical semigroup.
Suppose $g\in \bbf[S]$ with $g(0) \neq 0$ and $h\in \bbf[x]$ such that $gh\in \bbf[S]$.
Then $h\in \bbf[S]$.
\end{lem}

\begin{proof}
Suppose to the contrary that $h \notin \bbf[S]$.
Write $g = c_0 + c_1x + \cdots + c_mx^m$ and $h = c'_0 + \cdots + c'_nx^n$, and let
\[
d = \min\{j: c'_j \neq 0 \text{ and } j \notin S\}.
\]
Since $gh\in \bbf[S]$, the degree $d$ term of $gh$ is $0$.
On the other hand, we may express the degree $d$ term of $gh$ in terms of the $c_i$ and $c'_j$ to obtain
\[ \sum_{i+j = d} c_i c'_j = 0. \]
Since $S$ is additively closed, any pair $(i,j)$ appearing in the above sum must satisfy $i \notin S$ or $j \notin S$.
However, $c_i = 0$ for all $i\notin S$ and, by the minimality of $d$, we have $c'_j = 0$ for every $j \notin S$ with $j < d$.
As a result, the only nonzero term in the sum is the $(i,j) = (0,d)$ term, namely $c_0c'_d$.
Recalling that $c_0 = g(0) \neq 0$ and $c'_d \neq 0$ by definition, we obtain $c_0 c'_d \neq 0$, which is a contradiction.
\end{proof}

Together, the previous two lemmas allow us to determine the structure of a decomposition into products in $\Fq[x]$ of any irreducible element of $\Fq[S]$.

\begin{prop}\label{p:factorclassification}
Fix a numerical semigroup $S$ and an irreducible polynomial $f \in \Fq[S]$.
Then $f = x^m f_1\cdots f_k$, where
\begin{itemize}
\item $f_1,\dots, f_k\in \Fq[x]$ are irreducible with $f_i(0)\neq 0$ for each $i\in [1,k]$,
\item $0\le m < 2(\mathsf{F}(S)+1)$, and
\item $0 \le k \le q^{\mathsf{F}(S)}.$
\end{itemize}
\end{prop}

\begin{proof}
	For our later convenience and to match our earlier notation, we let $N = \mathsf{F}(S)+1$, so that every integer greater than $N$ lies in $S$.
	Suppose $f\in \Fq[S]$ is irreducible.  If $k = 0$, then $f = x^m$, so $m$ must be an atom in $S$, and thus $m \le 2(\mathsf F(S) + 1)$.  As such, suppose $k \ge 1$.  Moreover, if $f$ is irreducible in~$\Fq[x]$, then we are done, so suppose otherwise.

	\underline{Case 1}:\ If $x$ does not divide $f$, then we may write $f = f_1\cdots f_k$ for some irreducible polynomials $f_1,\dots,f_k\in \Fq[x]$ (and $f_i(0)\neq0$ for each $i$).
	Supposing $k> q^{\mathsf{F}(S)} = q^{N-1}$, by Lemma~\ref{lem:producing semigroup factor} there exists $g \in \Fq+x^N \Fq[x]$ with $g \mid f_1\dots f_{q^{N-1}}$.
	Let $h = (f_1\cdots f_k)/g$, and note that $h$ is nonconstant since $f_{q^{N-1}+1}\cdots f_k \mid h$.
	Since $gh = f_1\cdots f_k = f \in \Fq[S]$, we must have that $g(0) \neq 0$ (otherwise $f(0)=0$ and $f$ would be divisible by $x$).
	Now, by Lemma~\ref{lem:divisor closedness}, $h \in \Fq[S]$, which produces a contradiction to the irreducibility of $f$ in $\Fq[S]$ and implies that $k \le q^{\mathsf F(S)}$.

	\underline{Case 2}:\ If $f = x^m f_1\cdots f_k$ with $m > 0$ maximal (so that $x$ does not divide $f_1\cdots f_k$), then we need to show that $m<2(\mathsf{F}(S)+1) = 2N$ and that $k\leq q^{\mathsf{F}(S)}$.
	For the first part, if $m\ge 2N$ then we may write $f = x^{m-N} (x^N f_1\cdots f_k)$.
	Now we have produced a factorization of $f$ in $\Fq[S]$ as we know $x^{m-N}$ and $x^N (f_1\cdots f_k)$ both lie in $\Fq + x^N \Fq[x] \subseteq \Fq[S]$.
	The only remaining claim is the bound on $k$, so suppose $k > q^{N-1}$.
	Since $x$ does not divide $f_1 \cdots f_k$, we have that $f_i(0) \neq 0$ for each $i\le k$ and, as before, Lemma~\ref{lem:producing semigroup factor} yields a $g \in \Fq+x^N \Fq[x]$ with $g \mid f_1\cdots f_k$.
	Then, choosing $h\in \Fq[x]$ so that $gh = f_1\cdots f_k$, we now wish to show that $x^m h \in \Fq[S]$.
	Noting that $g (x^m h)\in \Fq[S]$ and $g(0) \neq 0$ (since $g \mid f_1\cdots f_k$), Lemma~\ref{lem:divisor closedness} implies that $x^m h\in \Fq[S]$.
	This yields a contradiction and we conclude, as in the previous case, that $k \leq q^{\mathsf{F}(S)}$.
\end{proof}

Additionally, the previous proposition
yields
a bound on the factorization length in $\Fq[x]$ of any irreducible element of $\Fq[S]$.

\begin{cor}
Fix a numerical semigroup $S$ and an irreducible polynomial $f \in \Fq[S]$. Any factorization of $f$ in $\Fq[x]$ has length at most
$q^{\mathsf{F}(S)}+2(\mathsf{F}(S)+1)$.
\end{cor}

\begin{remark}\label{r:boundimprovement}
The bounds in Proposition~\ref{p:factorclassification} are optimal when $q = 2$ and $S= \{0\} \cup \{m, m + 1, \ldots\}$ (that is, when $S$ is generated by an interval~\cite{nsalg2,fengraointerval}) but are not optimal in general.  It remains an interesting question to refine Proposition~\ref{p:factorclassification} based on the combinatorics of $S$; see Question~\ref{q:generaltypebreakdown2}.
\end{remark}

Before reaching our main goal for the section, we need the following auxiliary result, which gives a bound on certain sums over integer partitions in terms of the number of parts.

\begin{lem} \label{lem:log lemma}
If $n\ge k \ge 1$, then
\[
\sum_{\substack{m_1 \ge \cdots \ge m_k \ge 1 \\ n = m_1 + \cdots + m_k}} \frac{1}{m_1\cdots m_k} \le \frac{2^{k-1}\log^{k-1}(n)}{n},
\]
where the sum is taken over all partitions of $n$ into $k$ parts.
\end{lem}

\begin{proof}
The result trivially holds if $k=1$ because there is only one partition of $n$ into one part.
Proceeding by induction, suppose the lemma holds for a fixed $k\ge 1$.  Letting $M = \lfloor n/(k-1) \rfloor$, we~see

\begin{align*}
\sum_{\substack{m_1 \ge \cdots \ge m_{k+1} \ge 1 \\ n=m_1+\cdots+m_{k+1}}} \frac{1}{m_1\cdots m_{k+1}}
&\le \sum_{m=1}^M \frac{1}{m} \!\!\!\! \sum_{\substack{m_1 \ge \cdots \ge m_k \ge 1 \\ n-m=m_1+\cdots+m_k}} \!\!\!\! \frac{1}{m_1\cdots m_k}
\\
&\le \sum_{m=1}^M \frac{1}{m} \frac{2^{k-1}\log^k(n-m)}{n-m} \tag{inductive hypothesis} \\
&\le \sum_{m=1}^M \frac{1}{m} \frac{2^{k-1}\log^k(n)}{n-m} \\
&= 2^{k-1}\log^{k}(n) \sum_{m=1}^M \frac{1}{m(n-m)}.
\end{align*}

Now, since $1/x(n-x)$ is decreasing on the interval $(0,n/2)$ (and $M = \lfloor n/(k-1)\rfloor \le n/2$), a right Riemann sum of unit-width rectangles is an under approximation of the area under the graph from $x=0$ to $x=n/2$.
In particular, we may replace the sum in the last line with an integral to obtain
\begin{align*}
\sum_{m=1}^M \frac{1}{m(n-m)}
&\le \int_1^M \frac{1}{x(n-x)} \,dx \\
&= \frac{1}{n} \int_1^M \bigg( \frac{1}{x} + \frac{1}{n-x} \bigg) \, dx
\\
&= \frac{1}{n} \left( \int_1^M \frac{1}{x}\, dx + \int_{n-M}^{n-1} \frac{1}{u} \, du \right) \tag{letting $u=n-x$} \\
&= \tfrac{1}{n} \big( \log(M) - \log(1) + \log(n-1) - \log(n-M) \big) \\
&= \frac{1}{n} \log \bigg( \frac{M (n-1)}{n-M} \bigg) \\
&\le \frac{1}{n} \log(n^2) \tag{since $M \leq n$}\\
&= \frac{2}{n} \log(n).
\end{align*}
Finally, stringing together the centered inequalities above, we have
\[
\sum_{\substack{m_1 \ge \cdots \ge m_k{k+1} \\ n = m_1 + \cdots + m_{k+1}}} \frac{1}{m_1\cdots m_{k+1}}
\le 2^{k-1}\log^{k-1}(n) \bigg( \frac{2}{n} \log(n) \bigg)
= \frac{2^k \log^k(n)}{n}.
\]
Hence, we have the desired result.
\end{proof}

We briefly recall the following well-known result, whose proof is outlined in \cite[Section~14.3]{dummitfoote}.

\begin{thm}\label{t:polyatomicdensity}
Let $a_q(n)$ denote the number of irreducible polynomials of degree $n$ in $\Fq[x]$.  We have
\[
\frac{a_q(n)}{q^n} \le \frac{1}{n}
\]
for each $n \ge 2$.  In particular, $\Fq[x]$ has atomic density $0$.
\end{thm}

We are now ready to prove Theorem~\ref{t:atomicdensity}.

\begin{proof}[Proof of Theorem~\ref{t:atomicdensity}]
	Suppose $n \in \bbn$, and let $\mathsf g(S) = |\bbn \setminus S|$ denote the number of gaps of $S$.  In keeping with the notation from Theorem~\ref{t:polyatomicdensity}, let $a_q(n)$ denote the number of degree $n$ irreducible elements of $\Fq[x]$.
	Since we wish to calculate a limit as $n \to \infty$, we may assume $n > \mathsf{F}(S)$.
	Any~degree~$n$ polynomial $f \in \Fq[S]^{(n)}$ has the form $f = \sum_{i=0}^n c_i x^i$, where $c_n \in \Fq\setminus\{0\}$, $c_i = 0$ for all $i \in \bbn \setminus S$, and the remaining $c_i$ can be freely chosen from $\Fq$.
	Thus, we have
	\[
	|\Fq[S]^{(n)}| = (q-1)q^{n-\mathsf g(S)}.
	\]

By Proposition~\ref{p:factorclassification}, each irreducible $f \in \Fq[S]$ of degree $n$ can be written as $f = x^m f_1\cdots f_k$ in $\Fq[x]$ with $k < q^{\mathsf{F}(S)}$ and $\deg(f_1) + \cdots + \deg(f_k) = n-m$.
	Since $f$ is irreducible over $\Fq[S]$ then $m$ must be one of finitely many atoms in $S$ and so takes on finitely many values; namely, $m$ takes at most $2\mathsf{F}(S) + 2$ different values. We thereby bound $\rho_q^S(n)$ by
 \begin{align*}
	\rho_q^S(n)
	&=
	\frac{a_q^S(n)}{|\Fq[S]^{(n)}|}
	\le
	\frac{1}{|\Fq[S]^{(n)}|}\sum_{m=0}^{2\mathsf{F}(S) + 2} \underbrace{\sum_{\substack{m_1 \ge \cdots \ge m_k \\ n-m = m_1 +\cdots+ m_k}} a_q(m_1)\cdots a_q(m_k).}_{\substack{\textrm{Number of degree $n-m$ products} \\ \textrm{of $k$ irreducibles of $\Fq[x]$}}}
\end{align*}
For ease, let $M = q^{\mathsf{F}(S)}$.  As the above is a finite sum over the possible values of $m$, the fact that
\begin{align*}
    \frac{1}{|\Fq[S]^{(n)}|}
    \sum_{\substack{m_1 \ge \cdots \ge m_{k} \\ n=m_1+\cdots+m_{k}}}
	\!\!\!\!\!\!\!\!\!
	a_q(m_1) \cdots a_q(m_k)
	&\le
	\frac{1}{|\Fq[S]^{(n)}|}
	\sum_{k=1}^M \sum_{m_1,\dots,m_k} \!\! \left(\frac{q^{m_1}}{m_1}\right)\cdots \left(\frac{q^{m_k}}{m_k}\right) \tag{by Theorem \ref{t:polyatomicdensity}}
	\\
	&=
	\frac{1}{(q-1)q^{n-\mathsf g(S)}} \sum_{k=1}^M \sum_{m_1,\dots,m_k} \frac{q^n}{m_1\cdots m_k} \\
	&= \frac{q^{\mathsf g(S)}}{q - 1} \sum_{k=1}^M \sum_{m_1,\dots,m_k} \frac{1}{m_1\cdots m_k} \\
	&\le
	\frac{q^{\mathsf g(S)}}{q - 1} \sum_{k=1}^M \frac{\log^{k-1}(n)}{n} \tag{by Lemma~\ref{lem:log lemma}}
\end{align*}
tends to 0 as $n \to \infty$ implies $\rho_q^S(n) \to 0$ as well.  This completes the proof.
\end{proof}

\section{Irreducible polynomials over \texorpdfstring{$\fr$}{F2[x2,x3]}}
\label{sec:friendlyring}

The proof of Theorem~\ref{t:atomicdensity} relies on crude estimates on the number of each of the types of irreducible polynomials from the characterization given in Proposition \ref{p:factorclassification}.  In the specific setting of $\fr$, any irreducible element $f(x) \in \fr$ belongs to one of three classes of irreducible polynomials, which we call classic type, tame type, and wild type (see Definition~\ref{d:types}).
We determine each type by viewing irreducible elements of $\fr$ as (possibly reducible) elements of $\fx$ via the natural embedding
$\fr \hookrightarrow \fx$ and examining their factorization in $\fx$. In doing so, we use the type characterizations of irreducibles to provide a bound on the rate of convergence of the atomic density of $\fr$ (which we have already shown is zero by Theorem~\ref{t:atomicdensity}).

For the remainder of the paper, we consider only the case when $q=2$ and $S = \langle 2,3\rangle$.
Hence, we establish the shorthand notation
\[
a(n) = a_2^\bbn(n)
\qquad \text{and} \qquad
b(n)=a_2^{\langle 2,3 \rangle}(n)
\]
to denote the number of degree $n$ irreducible elements in $\bbf_2[x]$ and $\fr$, respectively.

We begin with the following proposition, which limits the possible ways an irreducible polynomial in $\fr$ can reduce in $\fx$.

\begin{prop}\label{p:lineartermfactors}
Let $f(x)\in\fr$ be a polynomial that is  irreducible in $\fr$ and  reducible in~$\fx$.  When reduced in $\fx$, each factor of $f(x)$ has a linear term. That is, if $f(x)=\ell_1(x)\cdots \ell_k(x)$ is the factorization of $f(x)$ in $\fx$ into irreducible terms, where
\[ \ell_i(x)=\sum_{j=0}^{\deg(\ell_i)} c_jx^j, \]
then $c_1\not=0$ for each $i$.
\end{prop}

\begin{proof}
Let $f(x)\in\fr$ be a polynomial that is irreducible in $\fr$ and  reduces in $\fx$.  Suppose we have the factorization
\[f(x)=\ell_1(x)\cdots \ell_k(x)q_1(x)\cdots q_m(x)\]
into irreducible terms, where $\ell_i(x) \in \fx \setminus \fr$ and $q_i(x)\in\fr$.  Let $r(x)=\ell_1(x)\cdots \ell_k(x)$ and $q(x)=q_1(x)\cdots q_m(x)$.  Then $f(x)=r(x)\cdot q(x)$, and $q(x) \in \fr$.  If $r(x)$ has no linear term, then $r(x)\in\fr$ and $f(x)$ factors into $r(x)$ and $q(x)$ in $\fr$.  This is a contradiction since $f(x)$ is irreducible in $\fr$.

It follows that $r(x)$ must have a nonzero linear term; that is, $r(x)\in\fx$ and $r(x)\not\in\fr$.  Since $q(x)$ has no linear term, it must be of the form $q(x)=1$,
$q(x)=x^{t_s}+\cdots+x^{t_2}+x^{t_1}$, or $q(x)=x^{t_s}+\cdots+x^{t_2}+x^{t_1}+1$, where $t_i>1$ for all $i$ and exponents are written in decreasing order.
If~$q(x)=x^{t_s}+\cdots+x^{t_2}+x^{t_1}=q_1(x)\cdots q_m(x)$,
then we can factor out at least one linear term, $x$.  This implies that at least one factor of $q(x)$, say $q_1(x)$, is in $\fx$ and not in $\fr$, a contradiction since all $q_i(x)\in\fr$.
If $q(x)=x^{t_s}+\cdots+x^{t_2}+x^{t_1}+1$, then $f(x)=r(x)\cdot q(x)$ when multiplied out has a nonzero linear term (specifically 1 times the nonzero linear term of $r(x)$) and so this contradicts that $f(x)\in\fr$.
Therefore, we conclude that $q(x)=1$.
\end{proof}

\begin{defn}\label{d:types}
Fix an irreducible polynomial $f(x) \in \fr$.  We say $f(x)$ is
\begin{itemize}
\item
of \emph{classic type} if $f(x)$ is irreducible in $\fx$;
\item
of \emph{tame type} if $f(x) = x^kg(x)$ where $g(x)$ is irreducible in $\fx$ with nonzero constant term and $k = 2$ or $3$; and

\item
of \emph{wild type} if $f(x) = g(x)h(x)$ where $g(x)$ and $h(x)$ are irreducible in $\fx$ and have nonzero constant terms.

\end{itemize}
We use $\bc$, $\bt$, and $\bw$  to denote the number of irreducible polynomials of degree $n$ in $\fr$ that are of classic, tame, and wild types, respectively. We observe that by Proposition~\ref{p:lineartermfactors}, each polynomial in $\fr$ falls into one of these three categories.
This implies
\[
b(n)=\bc+\bt+\bw.
\]
\end{defn}

\begin{ex}\label{e:types}
There exist five irreducible degree 4 polynomials in $\fr$, namely
\[
x^4 + x^3 + 1,
\qquad
x^4 + x^3 + x^2,
\qquad
x^4 + x^3,
\qquad
x^4+x^3+x^2+1,
\qquad
\text{and}
\qquad
x^4 + x^2 + 1.
\]
The first, $x^4 + x^3 + 1$, is of classic type, as it is irreducible in $\fx$ and has no linear term.
The next two, namely $x^4 + x^3 + x^2$ and $x^4 + x^3$, are both irreducible in $\fr$ but clearly factor in $\fx$, as we can factor out $x^m$ for some $m$.
Note that this is only possible if $2\leq m \leq 3$, as otherwise the original polynomial would reduce in $\fr$.  Lastly,
\[
x^4 + x^3 + x^2 + 1 = (x + 1)(x^3 + x + 1)
\qquad
\text{and}
\qquad
x^4 + x^2 + 1 = (x^2 + x + 1)^2
\]
both reduce in $\fx$, and all factors therein have a linear term and a nonzero constant term, making them both of wild type.
Table~\ref{tb:frdegree5} gives the classification of irreducible degree 5 polynomials in $\fx$.

\begin{table}[t]
\begin{center}
\begin{tabular}{|c|c|c|}
\hline
Type of Irreducible & Irreducible in $\fr$ & Factorization in $\fx$ \\
\hline
Classic & $x^5+x^4+x^3+x^2+1$ & $x^5+x^4+x^3+x^2+1$ \\
\hline
Classic & $x^5+x^3+1$ & $x^5+x^3+1$ \\
\hline
Tame & $x^5+x^4+x^3$ & $x^3(x^2+x+1)$ \\
\hline
Tame & $x^5+x^3+x^2$ & $x^2(x^3+x+1)$ \\
\hline
Wild & $x^5+1$ & $(x+1)(x^4+x^3+x^2+x+1)$ \\
\hline
Wild & $x^5+x^4+x^2+1$ & $(x+1)(x^4+x+1)$ \\
\hline
Wild & $x^5+x^2+1$ & $(x^2+x+1)(x^3+x+1)$ \\
\hline
\end{tabular}
\end{center}
\caption{Irreducible polynomials in $\fr$ of degree $5$.}
\label{tb:frdegree5}
\end{table}
\end{ex}

\begin{remark}\label{r:wildtype}
To highlight some of the nuances of factorization in $\fr$, let us consider a particularly interesting class of polynomials; namely, those of the form $x^p+1$ for some prime~$p$.

Note that in $\fx$,
we have the factorization $x^p+1=(x+1)(x^{p-1}+x^{p-2}+\cdots+x+1)$. Since wild type irreducible polynomials have factorizations of length 2 in $\fx$, and $x+1$ is always irreducible, it follows that $x^p+1$ is a wild type irreducible polynomial in $\fr$ if and only if $x^{p-1}+x^{p-2}+\cdots+x+1$, the $p^{th}$ cyclotomic polynomial $\Phi_p$, is irreducible in $\fx$. The number of factors of $\Phi_p$ in $\fx$ is given by $\phi(p)/ord_p(2)$, where $\phi$ is the Euler Totient function. It follows that $\Phi_p$ is irreducible if and only if $2$ is a primitive root modulo $p$. It is unknown for which primes $p$ that 2 is a primitive root as well as whether there are finitely or infinitely many such primes. A general discussion of when a number $a$ is a primitive root of $p$ can be found in \cite{leveque}. A conjecture of Artin implies that as $y \to \infty$, the ratio of primes $p \leq y$ for which $2$ is a primitive root of $p$ converges to $0.37456$. This result is dependent upon the extended Riemann hypothesis.

Though it is unknown for which primes $p$ we have
2 as a primitive root,
an examination for primes $2<p<1000$ revealed an interesting pattern for the factorization of $x^p+1$ in $\fr.$
For small $p$, we see that $x^p+1$ is reducible in $\fr$ if $p \equiv 1 \bmod 8$ or $p \equiv 7 \bmod 8$, and so in these cases $2$ is not a primitive root mod $p$ and $\Phi_p$ is reducible in $\fx$.
Thus, for small $p$ where $p \equiv 1 \bmod 8$ or $p \equiv 7 \bmod 8$, we have that $2$ is not a primitive root mod $p$ and $\Phi_p$ is reducible in $\fx$.
For example, $x^7+1$ and $x^{409}+1$ are reducible in $\fr$, noting that $409 \equiv 1 \bmod 8$, and so $\Phi_7$ and $\Phi_{409}$ are reducible in $\fx$.

This search also revealed that $p \equiv 3 \bmod 8$ or $p \equiv 5 \bmod 8$ may be a necessary but not sufficient condition for $x^p+1$ to be an irreducible
in $\fr$ and, subsequently, for $\Phi_p$ to be irreducible in $\fx$. We say not sufficient as there were primes $p$ for which $p \equiv 3 \bmod 8$ or $p \equiv 5 \bmod 8$ but $x^p+1$ was reducible in $\fr$. For example, $x^{131}+1$ is irreducible in $\fr$ and $131 \equiv 3 \bmod 8$, but $43 \equiv 3 \bmod 8$, yet $x^{43}+1$ is reducible in $\fr$.
For $2<p<1000$, there was no $x^p+1$ that is irreducible in $\fr$ when $p \equiv 1 \bmod 8$ or $p \equiv 7 \bmod 8$. However, for $2<p<1000$, when $x^p+1$ is irreducible in $\fr$ then $p \equiv 3 \bmod 8$ or $p \equiv 5 \bmod 8.$
\end{remark}

The above remark reveals the benefit of classifying irreducible polynomials in $\fr$ into three types, which we consider separately below.  We conclude this section with Theorem~\ref{t:friendlymain}, wherein we provide upper bounds on $\bc$, $\bt$, and $\bw$, using $a(j)$ for appropriate values of $j$, to obtain an upper bound on the convergence rate of the atomic density of $\fr$.

\begin{thm}\label{t:friendlymain}
In determining the atomic density of $\fr$, there is a constant $C$ such that
\[
\frac{b(n)}{2^n} \le \frac{\ln(n)}{n} + \frac{C}{n} + O\bigg(\frac{1}{n^2}\bigg).
\]
\end{thm}

\begin{proof}
Since $\bc \le a(n)$ for $n \geq 2$, we have
\[
\frac{\bc}{2^n}
\le \frac{a(n)}{2^n}
\le \frac{1}{n},
\]
by Theorem~\ref{t:polyatomicdensity}.  Similarly,
for $\bt$, we have
\begin{align*}
\frac{\bt}{2^n}
&\leq \dfrac{a(n - 2)}{2^n} + \dfrac{a(n - 3)}{2^n}
\leq \dfrac{1}{2^2(n - 2)} + \dfrac{1}{2^3(n - 3)}
\le \frac{1}{n}.
\end{align*}
This leaves $\bw$.
By Proposition~\ref{p:lineartermfactors}, each wild type irreducible polynomial in $\fr$ of degree~$n$ is a product of exactly two irreducible elements of $\fx$, which must have degrees~$k$ and $n-k$, respectively, for some $k = 1, \ldots, \lfloor n/2 \rfloor$.  In particular,
\[
\bw \le \sum_{k=1}^{\lfloor n/2 \rfloor} a(k)a(n-k).
\]
Applying Theorem~\ref{t:polyatomicdensity} in step three below, we obtain
\[
\frac{\bw}{2^n}
\le \sum_{k = 1}^{\lfloor n/2 \rfloor} \frac{a(k)a(n-k)}{2^n}
= \sum_{k = 1}^{\lfloor n/2 \rfloor} \frac{a(k)}{2^k}\frac{a(n-k)}{2^{n-k}}
\le \sum_{k = 1}^{\lfloor n/2 \rfloor} \frac{1}{k(n-k)},
\]
which can be simplified as
\begin{align*}
\sum_{k = 1}^{\lfloor n/2 \rfloor} \frac{1}{k(n-k)}
\frac{1}{n}\sum_{k=1}^{\lfloor n/2 \rfloor} \bigg(\frac{1}{k} + \frac{1}{n-k}\bigg)
&= \frac{1}{n}\sum_{k=1}^{\lfloor n/2 \rfloor} \frac{1}{k} + \frac{1}{n}\sum_{k=1}^{\lfloor n/2 \rfloor} \frac{1}{n-k}
\\
&= \frac{1}{n}\sum_{k=1}^{\lfloor n/2 \rfloor} \frac{1}{k} + \frac{1}{n}\sum_{k=\lceil n/2 \rceil}^{n-1} \frac{1}{k}
\\
&\le \frac{2}{n} + \frac{1}{n}\sum_{k=1}^{n-1} \frac{1}{k},
\end{align*}
with equality met precisely when $n$ is odd.
Finally, the asymptotic growth rate of the harmonic series \cite[Theorem~6.10]{
leveque}
 yields, for some constant $C$ dependent on the Euler-Mascheroni constant,
\[
\frac{b(n)}{2^n}
= \frac{\bc + \bt + \bw}{2^{n}}
\le \dfrac{4}{n} + \frac{1}{n}\sum_{k=1}^{n-1} \frac{1}{k}
\le \frac{\ln(n)}{n} + \dfrac{C}{n} + O\bigg(\frac{1}{n^2}\bigg),
\]
thereby completing the proof.
\end{proof}

\section{Counting irreducible polynomials by degree in \texorpdfstring{$\fr$}{F2[x2,x3]}}
\label{sec:counting}

The number of monic irreducible polynomials of degree $n$ over a finite field $\mathbb{F}_q$ is given by Gauss' formula
\[a_q(n)=\displaystyle \frac{1}{n} \sum_{d|n} \mu\big(n/d\big) \cdot q^d,\]
where $a_q(n)$ denotes the number of atoms in $\Fq[x]$ (in the parlance of Theorem \ref{t:polyatomicdensity}) and $\mu(n)$ is the M\"{o}bius function. A proof of this formula can be found in \cite{chebolu}; the formula originally appears in \cite{gauss}.  In Section~\ref{sec:friendlyring}, we used $a(n)$ to bound the number of classic, tame, and wild irreducible polynomials in $\fr$.
This allowed us to bound the rate of convergence of the atomic density of $\fr$.
In this section, we give an explicit formula for $b(n)$, the number of (monic) irreducible polynomials of degree~$n$ in $\fr$, that also relies on the M\"obius function.

\begin{lem}\label{l:frcount}
The number $b(n)$ of irreducible polynomials in $\fr$ of degree $n$ is given by the expression $b(n) = \bc+\bt+\bw$, where
\begin{align*}
\bc&=a(n)-s(n), \\
\bt&=s(n-2)+s(n-3), \\
\bw&=\begin{cases}
\displaystyle
\!\sum_{k=1}^{n/2 - 1} s(k) s(n-k)
+ \binom{s(n/2)+1}{2} & \text{if $n$ is even;} \\
\displaystyle\sum_{k=1}^{\lfloor n/2 \rfloor} s(k) s(n-k) & \text{if $n$ is odd,}
\end{cases}
\end{align*}
$a(n)$ is the number of irreducible polynomials in $\fx$ of degree $n$, and $s(n)$ is the number of such polynomials whose linear and constant terms are both nonzero.
\end{lem}

\begin{proof}
 We determine $b(n)$ by considering the number of each type of irreducible polynomial in $\fr$, as outlined in Definition~\ref{d:types}. First, note that a classic type irreducible polynomial of degree $n$ in $\fr$ is also irreducible in $\fx$, so it must belong to the set of irreducible polynomials in $\fx$ of degree $n$ without a linear term. It follows immediately that $\bc = a(n) - s(n)$.

We can construct tame type and wild type irreducible polynomials in $\fr$ using irreducible polynomials with a linear term and nonzero constant term in $\fx$. Recall that a tame type irreducible polynomial of degree $n$ in $\fr$ factors in $\fx$ as $x^tg(x)$, where $2 \leq t \leq 3$ and $g(x)$ is an irreducible in $\fx$ of degree $n-t$ with a nonzero linear and constant term. All tame type irreducible polynomials of degree $n$ in $\fr$ are of the form $f(x)=x^tg(x)$ where $g(x)$ runs through each monic irreducible polynomial of degree $n-t$ for $t=2$ and $t=3$. Since $\fx$ is a unique factorization domain, there is a one to one correspondence between the irreducible polynomials $g(x)$ in $\fx$ enumerated by $s(n-t)$ and the set of tame type irreducible polynomials in $\fr$.
 Note that the cardinality of the tame type irreducible polynomials that factor as $x^tg(x)$ with $t=2$ in $\fx$ is equal to $s(n-2)$, and similarly the cardinality of the tame type irreducible polynomials that factor as $x^tg(x)$ with $t=3$ in $\fx$ is equal to $s(n-3)$. More explicitly, we have that
$\bt=s(n-2)+s(n-3)$, where $\bt$ is the total number of tame type irreducible polynomials of degree $n$ in $\fr$.

Lastly, note that a wild type irreducible polynomial of degree $n$ in $\fr$ factors in $\fx$ as $g_1(x)g_2(x)$, where $g_1(x)$ has degree $1 \le k < n$ and and $g_2(x)$ has degree $n-k$. Recall that both $g_1(x)$ and $g_2(x)$ must be atoms in $\fx$, each with a linear term and nonzero constant term. To then build wild type irreducible polynomials in $\fr$, we begin with a monic irreducible polynomial of degree $k$ with $1\leq k <n$ in $\fx$ with a linear term and nonzero constant term, call it $g_1(x)$. We~then multiply $g_1(x)$ by a monic irreducible polynomial $g_2(x)$ of degree $n-k$ with a linear term and nonzero constant term. Running through all possible pairs of monic irreducible polynomials $g_1(x)$ and $g_2(x)$ of degree $1\leq k <n$ and $n-k$, respectively, we see that the number of wild type irreducible polynomials in $\fr$ is given by
\[
\bw=\sum_{k=1}^{\lfloor n/2 \rfloor} s(k) s(n-k),
\]
in the case that $n$ is odd.  When $n$ is even, we obtain
\[
\bw = \sum_{k=1}^{\lfloor \frac{n-1}{2}\rfloor} s(k) s(n-k)
+ \binom{s(n/2)}{2} + s(n/2)
= \sum_{k=1}^{n/2 - 1} s(k) s(n-k)
+ \binom{s(n/2)+1}{2},
\]
where the summation comes from the pairs $g_1(x)$, $g_2(x)$ with $\deg(g_1(x)) \neq \deg(g_2(x))$ and the remainder comes from the pairs with $\deg(g_1(x)) = \deg(g_2(x))$.
This completes the proof.
\end{proof}

\begin{ex}\label{e:frlist}
Table~\ref{tb:fxdegree2345} shows the irreducible polynomials in $\fx$ of degree $d \leq 5$ with and without a linear term. We can use these to construct all irreducible polynomials of degree 5 in $\fr$.

\begin{table}[t]
\begin{center}
\begin{tabular}{|c|c|c|c|}
\hline
Deg. & $s(d)$
& Irreducible polynomials & Irreducible polynomials \\
&& with linear term & without linear term \\
\hline
2 & 1
&  $x^2+x+1$ & \\
\hline
3 & 1
&  $x^3+x+1$ & $x^3+x^2+1$ \\
\hline
4 & 2
&  $x^4+x^3+x^2+x+1$ & $x^4+x^3+1$ \\
&& $x^4+x+1$ &  \\
\hline
5 & 3
&  $x^5+x^4+x^3+x+1$ & $x^5+x^4+x^3+x^2+1$  \\
&& $x^5+x^4+x^2+x+1$ & $x^5+x^3+1$ \\
&& $x^5+x^3+x^2+x+1$ & \\
\hline
\end{tabular}
\end{center}
\caption{Irreducible polynomials in $\fx$ of degree $d = 2, 3, 4, 5$.}
\label{tb:fxdegree2345}
\end{table}

To find classic type irreducible polynomials of degree 5 in $\fr$, we take all irreducible polynomials in $\fx$ of degree 5 without a linear term; that is, $x^5+x^4+x^3+x^2+1$ and $x^5+x^3+1$. We see that $\bcn{5}=a(5)-s(5)=5-3=2$.

Next, we note that a tame type irreducible polynomial of degree 5 in $\fr$ is of the form $x^3g(x) \in \fx$ or $x^2g(x) \in \fx$, where $g(x)$ is an irreducible polynomial of degree 2 with a linear term or of degree 3 with a linear term, respectively. Thus, the tame type irreducible polynomials in $\fr$ are $x^5+x^4+x^3=x^3(x^2+x+1)$ and $x^5+x^3+x^2=x^2(x^3+x+1)$, and we see that $\btn{5}=s(3)+s(2)=1+1=2$.

Lastly, we build wild type irreducible polynomials of degree 5 in $\fr$ by multiplying pairs of irreducible polynomials with a linear term whose degree add to 5. We have the following possibilities: $(x+1)(x^4+x^3+x^2+x+1)=x^5+1$, $(x+1)(x^4+x+1)=x^5+x^4+x^2+1$, and $(x^2+x+1)(x^3+x+1)=x^5+x^2+1$. Thus, the wild type irreducible polynomials of degree 5 in $\fr$ are $x^5+1, x^5+x^4+x^2+1,$ and $x^5+x^2+1$, and we see $\bwn{5}=s(1)s(4)+s(2)s(3)=1\cdot 2 + 1\cdot 1=3$~and
\[
b(5)=\bcn{5}+\btn{5}+\bwn{5}=a(5)-s(5)+s(3)+s(2)+s(1)s(4)+s(2)s(3)=3+2+3=8.
\]
All of the irreducible polynomials of degree 5 are given in Table~\ref{tb:frdegree5}.
\end{ex}

Lemma~\ref{l:frcount} and Example~\ref{e:frlist} highlight that we can both build and enumerate irreducible polynomials of a given degree in $\fr$ if we have enough knowledge about the irreducible polynomials in $\fx$. While Lemma 4.1 gives an explicit formula for counting irreducibles of a given degree $n$ in $\fr$, it is reliant upon knowing $s(n)$ for arbitrary $n$. We next show that we can determine the number of irreducible polynomials with a linear term of a given degree $n$ in $\fx$ by displaying a bijection between irreducible polynomials in this set and self-irreducible polynomials of degree $2n$.

Let $f(x) \in \mathbb{F}_q[x]$ be given by $f(x)=c_nx^n+c_{n-1}x^{n-1}+\cdots+c_1x+c_0$. The \emph{reciprocal} polynomial of $f(x)$, denoted $f^*(x)$, is given by $f^*(x)=x^nf\big(\frac{1}{x}\big)=c_0x^n+c_1x^{n-1}+\cdots+c_{n-1}x+c_n.$ If $f(x)=f^*(x)$, then $f(x)$ is said to be a \emph{self-reciprocal polynomial} (or a \emph{palindrome}).
Any polynomial $f(x)$ of degree $n$ over $\mathbb{F}_q$ can be transformed into a self-reciprocal polynomial $f^Q(x)=x^nf(x+\frac{1}{x})$ of degree $2n$.

Conversely, the~following theorem states that every self-reciprocal polynomial of degree $2n$ lies in the image of this transformation.

\begin{thm}[{\cite[Lemma~2.75 in Ch.~7]{jungnickel}}]\label{t:selfreciprocaltransform}
Let $g(x)$ be any monic self-reciprocal polynomial of degree~$2n$ in $\Fq[x]$.  Then there exists a polynomial $f(x)$ of degree $n$ in $\Fq[x]$ such that $g(x) = f^Q(x)$. If $g(x)$ is irreducible, then $f(x)$ is also irreducible.
\end{thm}

This theorem implies that the map from the set of monic irreducible polynomials of degree $n$ in $\mathbb{F}_q[x]$ to the set of monic irreducible self-reciprocal polynomials of degree $2n$ over $\mathbb{F}_q$ is surjective.
However, if we start with a monic irreducible polynomial of degree $n$ over $\mathbb{F}_q$, then $f^Q(x)$ is a monic self-reciprocal polynomial but is not necessarily irreducible. For example, $f(x)=x^5+x^3+1$ is irreducible over $\fx$ but $f^Q(x)=x^{10}+x^6+x^5+x^4+1=(x^5+x^4+x^2+x+1)^2$ is not irreducible over $\fx$. The following results answer the question of what restrictions must be placed on a monic irreducible polynomial $f(x)$ so that $f^Q(x)$ is irreducible. In particular, the next result originally appeared in \cite{varsamov}; we provide another reference for ease of access.

\begin{thm}[{\cite[Theorem~6 and Corollary~7]{meyn}}]\label{t:meyn}
Let $f(x)=c_nx^n+c_{n-1}x^{n-1}+\cdots+c_1x+c_0$ be a monic irreducible polynomial of degree $n \neq 1$ over $\bbf_q$, where $q$ is even. Then $f^Q(x)$ is also irreducible if and only if $Tr(f(x)) = c_1/c_0 = 1$.  In particular, if $q = 2$, then $f^Q(x)$ is irreducible if and only if $c_1 = 1$.
\end{thm}

Theorem~\ref{t:meyn} gives us the precise connection that we want.  We restate it in the following corollary.

\begin{cor}\label{c:irredreciprocal}
The number of irreducible self-reciprocal polynomials in $\fx$ of degree $2n$ is equal to the number of irreducible polynomials if $\fx$ of degree $n$ with a linear term.
\end{cor}

Note that by Corollary~\ref{c:irredreciprocal}, we can replace each $s(n)$ term in Lemma~\ref{l:frcount} with $r_2(n)$, which represents the number of self-reciprocal irreducible polynomials of degree $2n$ over $\fx$.  While $s(n)$ is not generally known for arbitrary $n$, there is a known formula for $r_2(n)$ using M\"obius inversion.

\begin{lem}[{\cite[Theorems~2 and~3]{carlitz}}] \label{l:carlitz}
The number of monic self-reciprocal irreducible polynomials of degree $2n$ over $\bbf_q$ is given by
\[
r_q(n) = \begin{cases}
\displaystyle \frac{1}{2n}(q^n-1) & \text{ if $q$ is odd and $n=2^d$;} \\[0.5em]
\displaystyle \frac{1}{2n} \sum_{\substack{d \mid n \\ d \text{ odd}}} \mu(d)q^{n/d} & \text{ otherwise,}
\end{cases}
\]
where $\mu$ is the M\"{o}bius function.
\end{lem}

\begin{thm}
There is an algebraic expression for the number $b(n)$ of irreducible polynomials in $\fr$ of degree $n$ in terms of the M\"{o}bius function.
\end{thm}

\begin{proof}
Substitute Corollary~\ref{c:irredreciprocal} and Lemma~\ref{l:carlitz} into Lemma~\ref{l:frcount}.
\end{proof}

\section{Future Work} \label{future}

We provide several problems that indicate possible extensions of this work. These problems fall into two categories: those that focus on further exploring details in the current setting of numerical semigroup algebras and those that focus on expanding the results of this paper beyond the current setting of numerical semigroup algebras.

In Theorem \ref{t:atomicdensity}, we proved that any numerical semigroup algebra $\Fq[S]$ has atomic density 0 by splitting the irreducible elements of $\Fq[S]$ into finitely many classes based on their factorizations in~$\Fq[x]$. However, the estimates used for how many irreducible polynomials are in each class were unrefined in the general setting of $\Fq[S]$.
We were able to obtain more refined results for the semigroup algebra $\fr$ but it remains an interesting task to do this in the broader setting of numerical semigroup algebras.

\begin{problem}\label{q:generaltypebreakdown2}
Classify the types of irreducible polynomials in more general numerical semigroup algebras, and provide a tighter bound on the number irreducible polynomials of each degree therein.
\end{problem}

Noting that the atomic density of $\Fq[x]$ is known to be 0, a natural next setting in which to study atomic density is one in which there exists a natural ordering to the elements of the semigroup so that there exists a natural value to tend to infinity. There is not such a natural ordering for all semigroups, but there is for numerical semigroups. Can we obtain results similar to those in this paper in other settings where there exists one or more natural orderings of the elements?

\begin{problem}\label{q:affinealg2}
Investigate atomic density of finitely generated semigroup algebras
over finite fields.
\end{problem}

In stark contrast to the univariate case, it is known that, when ordered by total degree, the multivariate polynomial ring $\bbf[x_1, \ldots, x_k]$ over a finite field $\bbf$ has atomic density 1~\cite{carlitz1963distribution}.  As such, we~conjecture the following.

\begin{conjecture}\label{q:affine}
Any affine semigroup algebra $\bbf[S]$ over a finite field $\bbf$ has atomic density 1.
\end{conjecture}

Factorization invariants are widely studied for families of rings and semigroups~\cite{nonuniq}, though little is known in the context of semigroup algebras over finite fields.

\begin{problem}\label{q:elasticity}
Assume the atomic density of a semigroup algebra is zero. What conclusions can be made about other factorization invariants, such as elasticity?  What happens when the atomic density is nonzero?
\end{problem}

We can also consider atomic density in the ring of formal power series $R[[X]]$, where~$R$ is a commutative ring. In this case, it is of interest to understand how the ring $R$ affects the atomic density.

\begin{problem} \label{q:pwrseries}
Determine the atomic densities of formal power series rings $\mathbb F_q[\![S]\!]$ and $\mathbb F_q[S][\![y]\!]$, where~$S$ is a numerical semigroup.
\end{problem}

Thus far, only finite fields have been considered, in part because coefficients can then be chosen uniformly.  To consider this question over infinite fields, one needs to choose a probability distribution on the base field.

\begin{problem}\label{q:infintefield}
Given a numerical semigroup $S$ and a probability distribution on $\mathbb Q$, find the atomic density of $\mathbb Q[S]$.
\end{problem}

It is a well known fact that each element of $\overline{\mathbb F}_p$, the algebraic closure of $\mathbb F_p$, is an element of $\mathbb F_{p^r}$ for some $r \ge 1$. As such, we suspect the following holds for any probability distribution on $\overline{\mathbb F}_p$.

\begin{conjecture}\label{q:algclosure}
Given a numerical semigroup $S$ and a prime $p$, the atomic density of $\overline{\mathbb F}_p[S]$ is 0.
\end{conjecture}

We close by reiterating a problem originally found in~\cite{acmperiodic}.

\begin{problem}\label{q:acm}
Determine the atomic density of the arithmetical congruence monoid $M_{a,b}$.
\end{problem}

\bibliographystyle{amsplain}
\bibliography{bibliography}

\providecommand{\bysame}{\leavevmode\hbox to3em{\hrulefill}\thinspace}
\providecommand{\MR}{\relax\ifhmode\unskip\space\fi MR }
\providecommand{\MRhref}[2]{%
  \href{http://www.ams.org/mathscinet-getitem?mr=#1}{#2}
}
\providecommand{\href}[2]{#2}
\begin{thebibliography}{10}

\bibitem{nsalg1}
David~F Anderson, Scott Chapman, Faith Inman, and WW~Smith, \emph{Factorization
  in $k[x^2,x^3]$}, Archiv der Mathematik \textbf{61} (1993), no.~6, 521--528.

\bibitem{nsalg2}
David~F Anderson and Susanne Jenkens, \emph{Factorization in $k[x^n,
  x^{n+1},\ldots,x^{2n-1}]$}, Communications in Algebra \textbf{23} (1995),
  no.~7, 2561--2576.

\bibitem{markovbook}
Satoshi Aoki, Hisayuki Hara, and Akimichi Takemura, \emph{Markov bases in
  algebraic statistics}, vol. 199, Springer Science \& Business Media, 2012.

\bibitem{numericalappl}
Abdallah Assi and Pedro~A Garc{\'\i}a-S{\'a}nchez, \emph{Numerical semigroups
  and applications}, vol.~1, Springer, 2016.

\bibitem{baginski2014arithmetic}
Paul Baginski and Scott Chapman, \emph{Arithmetic congruence monoids: a
  survey}, Combinatorial and Additive Number Theory, Springer, 2014,
  pp.~15--38.

\bibitem{baruccinsalg}
Valentina Barucci, \emph{Numerical semigroup algebras}, Multiplicative ideal
  theory in commutative algebra, Springer, 2006, pp.~39--53.

\bibitem{berlekampcodingtheory}
Elwyn Berlekamp, \emph{Algebraic coding theory}, World Scientific, 1968.

\bibitem{carlitz}
L~Carlitz, \emph{Some theorems on irreducible reciprocal polynomials over a
  finite field}, J. reine angew. Math \textbf{227} (1967), no.~212-220, 12.

\bibitem{carlitz1963distribution}
Leonard Carlitz et~al., \emph{The distribution of irreducible polynomials in
  several indeterminates}, Illinois Journal of Mathematics \textbf{7} (1963),
  no.~3, 371--375.

\bibitem{chebolu}
Sunil~K Chebolu and J{\'a}n Min{\'a}{\v{c}}, \emph{Counting irreducible
  polynomials over finite fields using the inclusion-exclusion principle},
  Mathematics Magazine \textbf{84} (2011), no.~5, 369--371.

\bibitem{designtestgeneration}
David~M Cohen, Siddhartha~R Dalal, Jesse Parelius, and Gardner~C Patton,
  \emph{The combinatorial design approach to automatic test generation}, IEEE
  software \textbf{13} (1996), no.~5, 83--88.

\bibitem{clo}
David Cox, John Little, and Donal OShea, \emph{Ideals, varieties, and
  algorithms: an introduction to computational algebraic geometry and
  commutative algebra}, Springer Science \& Business Media, 2013.

\bibitem{cls}
David~A Cox, John~B Little, and Henry~K Schenck, \emph{Toric varieties}, vol.
  124, American Mathematical Soc., 2011.

\bibitem{coykendall}
Jim Coykendall and Felix Gotti, \emph{On the atomicity of monoid algebras},
  Journal of Algebra \textbf{539} (2019), 138--151.

\bibitem{fengraointerval}
Manuel Delgado, J~Farr{\'a}n, P~Garc{\'\i}a-S{\'a}nchez, and David Llena,
  \emph{On the generalized feng-rao numbers of numerical semigroups generated
  by intervals}, Mathematics of Computation \textbf{82} (2013), no.~283,
  1813--1836.

\bibitem{numericalsgpsgap}
Manuel Delgado and Pedro~A Garc{\'\i}a-S{\'a}nchez, \emph{numericalsgps, a gap
  package for numerical semigroups}, ACM Communications in Computer Algebra
  \textbf{50} (2016), no.~1, 12--24.

\bibitem{algmarkov}
Persi Diaconis, Bernd Sturmfels, et~al., \emph{Algebraic algorithms for
  sampling from conditional distributions}, The Annals of statistics
  \textbf{26} (1998), no.~1, 363--397.

\bibitem{dummitfoote}
David~Steven Dummit and Richard~M Foote, \emph{Abstract algebra}, vol.~3, Wiley
  Hoboken, 2004.

\bibitem{gauss}
Carl~Friedrich Gauss, \emph{Untersuchungen \"{u}ber h\"{o}here {A}rithmetik},
  Deutsch herausgegeben von H. Maser, Chelsea Publishing Co., New York, 1965.
  \MR{0188045}

\bibitem{krullcombinatorialsurvey}
Alfred Geroldinger, \emph{Krull domains and monoids, their sets of lengths, and
  associated combinatorial}, Factorization in integral domains \textbf{189}
  (1997), 73.

\bibitem{nonuniq}
Alfred Geroldinger and Franz Halter-Koch, \emph{Non-unique factorizations:
  Algebraic, combinatorial and analytic theory}, CRC Press, 2006.

\bibitem{acmperiodic}
Jacob Hartzer and Christopher O'Neill, \emph{On the periodicity of irreducible
  elements in arithmetical congruence monoids}, Integers \textbf{17} (2017),
  Paper No. A38, 7. \MR{3704579}

\bibitem{jungnickel}
Dieter Jungnickel, \emph{Finite fields: structure and arithmetics},
  BI-Wiss.-Verl., 1993.

\bibitem{leveque}
William~Judson LeVeque, \emph{Fundamentals of number theory}, Courier
  Corporation, 1996.

\bibitem{sparsesensingfinitegeometry}
Shuxing Li and Gennian Ge, \emph{Deterministic construction of sparse sensing
  matrices via finite geometry}, IEEE transactions on signal processing
  \textbf{62} (2014), no.~11, 2850--2859.

\bibitem{ecctheory1}
Florence~Jessie MacWilliams and Neil James~Alexander Sloane, \emph{The theory
  of error correcting codes}, vol.~16, Elsevier, 1977.

\bibitem{meyn}
Helmut Meyn and Werner G{\"o}tz, \emph{Self-reciprocal polynomials over finite
  fields}, 1990.

\bibitem{numericalsgpssage}
Christopher O'Neill, \emph{{numsgps-sage (Sage software)}}, 2020.

\bibitem{numsgpsalg}
Christopher O'Neill and Sviatoslav Zinevich, \emph{{numsgpsalg (Sage
  software)}}, 2019.

\bibitem{numericalsurvey}
Christopher O’Neill and Roberto Pelayo, \emph{Factorization invariants in
  numerical monoids}, Algebraic and Geometric Methods in Discrete Mathematics
  \textbf{685} (2017), 231.

\bibitem{lensetprogress}
Wolfgang~A Schmid, \emph{Characterization of class groups of krull monoids via
  their systems of sets of lengths: a status report}, Number Theory and
  Applications, Springer, 2009, pp.~189--212.

\bibitem{grobpoly}
Bernd Sturmfels, \emph{Grobner bases and convex polytopes}, vol.~8, American
  Mathematical Soc., 1996.

\bibitem{varsamov}
RR~Varshamov and GA~Garakov, \emph{On the theory of selfdual polynomials over a
  galois field}, Bull. Math. Soc. Sci. Math. RS Roumanie (NS) \textbf{13}
  (1969), 403--415.

\end{thebibliography}

\end{document}